\documentclass[11pt]{article}

\usepackage[english]{babel}
\usepackage{amssymb}
\usepackage{amsmath}
\usepackage{hyperref}
\usepackage{enumerate}
\usepackage{amsthm}
\usepackage{amsfonts}
\usepackage{listings}
\usepackage{multicol}
\usepackage[curve]{xypic}

\usepackage{MnSymbol}
\usepackage{stmaryrd}

\DeclareMathSymbol{\mlq}{\mathord}{operators}{``}
\DeclareMathSymbol{\mrq}{\mathord}{operators}{`'}

\title{A note on formal periods}
\date{July 24, 2021}
\author{Fritz H\"ormann\\ Mathematisches Institut, Albert-Ludwigs-Universit\"at Freiburg}

\usepackage[numbers,sort&compress]{natbib}

\usepackage{color}

\usepackage{tikz}

\newtheorem{SATZ}{Theorem}[section]

\newtheorem{LEMMA}[SATZ]{Lemma}
\newtheorem{KEYLEMMA}[SATZ]{Key Lemma}
\newtheorem{DEF}[SATZ]{Definition}
\newtheorem{PROP}[SATZ]{Proposition}

\newtheorem{BEISPIEL}[SATZ]{Example}
\newtheorem{FRAGE}[SATZ]{Question}

\newtheorem{KOR}[SATZ]{Corollary}

\newtheorem{BEM}[SATZ]{Remark}

\newtheoremstyle{bare}        
  {}            
  {}            
  {\normalfont}                 
  {}                            
  {\bfseries}                   
  {}                            
  {.0em}                           
  {\thmnumber{#2}#1. \thmnote{\normalfont\textsc{(#3)}} } 

\theoremstyle{bare}
\newtheorem{PAR}[SATZ]{}

\newcommand{\comment}[1]{}

\newcommand{\C}{ \mathbb{C} }
\newcommand{\Q}{ \mathbb{Q} }

\newcommand{\Z}{ \mathbb{Z} }

\newcommand{\Gm}{ {\mathbb{G}_m} }

\newcommand{\OOO}{\text{\footnotesize$\mathcal{O}$}}

\DeclareMathOperator{\colim}{colim}

\DeclareMathOperator{\id}{id}

\DeclareMathOperator{\op}{op}

\DeclareMathOperator{\coker}{coker}
\DeclareMathOperator{\Hom}{Hom}

\DeclareMathOperator{\End}{End}

\DeclareMathOperator{\im}{im}
\DeclareMathOperator{\eval}{eval}
\DeclareMathOperator{\tr}{tr}

\newcommand{\cat}[1]{ {[\textnormal{ \textbf{#1} }]} }
\newcommand{\Qbar}{{\overline{\mathbb{Q}}}}

\begin{document}

\maketitle



\section*{Abstract}

We give an elementary description of the space of  formal periods of a mixed motive. This allows for a simplified reformulation of the period conjectures of Grothendieck and Kontsevich-Zagier. Furthermore, we develop a machinery which in principle allows to determine the space of formal periods for an arbitrary mixed motive explicitly. 

\section{Introduction}

Let $\mathcal{A}$ be a $\Q$-linear Abelian category with two (covariant) fiber functors\footnote{i.e.\@ $\Q$-linear, exact and faithful}
\[ H_{B} : \mathcal{A} \rightarrow \cat{f.d.-$\Q$-Vect} \qquad H_{dR} : \mathcal{A} \rightarrow \cat{f.d.-$\Qbar$-Vect}  \]
to the category of finite dimensional $\Q$ (resp.\@ $\Qbar$) vector spaces. 
Although the following discussion is completely abstract, we imagine that $\mathcal{A}$ is a (sub-)category of mixed motives, and that the fiber functors correspond to Betti homology, and (the dual of) de Rham cohomology, respectively.

The $\Qbar$-vector space of {\bf formal periods} is the quotient of
\begin{equation}\label{eqBdR}
 \bigoplus_{M \in \mathcal{A}} H_B(M) \otimes_\Q H_{dR}(M)^\vee   
\end{equation}
modulo the relations
\[ (\sigma \otimes \alpha^\vee(\omega))_M = (\alpha(\sigma) \otimes \omega)_N \]
for any morphism $\alpha: M \rightarrow N$ in $\mathcal{A}$, $\sigma \in H_{B}(M)$, and $\omega \in H_{dR}(N)^\vee$.

A {\em concrete} theory of periods is given by an isomorphism
\[  \int: H_{B} \otimes_\Q \C \cong H_{dR} \otimes_\Qbar \C  \]
of fiber functors. Of course, any reasonable category of mixed motives must come equipped with such an isomorphism induced by the integration of forms. 
It induces a $\Qbar$-linear morphism
\begin{eqnarray*} 
\eval_\int: P(\mathcal{A}) &\rightarrow& \C \\
(\sigma \otimes  \omega)_M &\mapsto& \int_\sigma \omega  := \omega(\int_M(\sigma)) \quad \text{(suggestive notation)} .
\end{eqnarray*} 

\newpage

Consider the following properties:
\begin{itemize}
\item[(A)] $\eval_\int$ is  injective.
\item[(B)] For all $M \in \mathcal{A}$, $\sigma \in H_B(M)$, and $\omega \in H_{dR}(M)^\vee$ such that 
\[ \int_\sigma \omega = 0 \]
there is an exact sequence
\[ \xymatrix{ 0 \ar[r] & N' \ar[r] & M \ar[r] & N \ar[r] & 0 } \]
with $\sigma \in H_B(N')$ and $\omega \in H_{dR}(N)^\vee$.
\end{itemize}
If (A) holds, we say that the {\bf period conjecture} holds for $(\mathcal{A}, H_{B}, H_{dR}, \int)$.
For example, Grothendieck's period conjecture (+ connectedness of the torsor, see below) and the period conjecture of Kontsevich-Zagier \cite{KZ01} are equivalent to the period conjecture for
Nori motives $(\cat{Nori-Mot${}_\Qbar$}, H_{B}, H_{dR}, \int)$. See Huber and M\"uller-Stach \cite[Proposition 13.2.6]{HMS17} and Ayoub \cite[Corollary 32]{Ayo14}. 
We refer to \cite{And04, HMS17, Hub20, Ayo14, BC16}  for a general discussion of these conjectures and their relation to existing theories of mixed motives.
Not much is known in general, except in small dimensions, where we have the following theorem of Huber and W\"ustholz based on previous transcendence results by Baker, Gelfond, Schneider, and W\"ustholz:
\begin{SATZ}[\cite{HW18}] \label{SATZHW}
The period conjecture holds for $(\cat{1-Mot${}_\Qbar$}, H_{B}, H_{dR}, \int)$ where $\cat{1-Mot${}_\Qbar$}$ is the $\Q$-linear category of Deligne 1-motives \cite{Del74} defined over $\Qbar$ equipped with its natural fiber functors. 
\end{SATZ}

The main observation of this note (for an arbitrary Abelian category with fiber functors and comparison isomorphism as above) is
\begin{SATZ}[see Corollary~\ref{KORAB}]\label{SATZPCALT}
Properties (A) and (B) are equivalent. 
\end{SATZ}

Actually in \cite[Theorem 3.6]{HW18} the period conjecture is shown for the Abelian category of 1-motives precisely by establishing (B) --- the statement of \cite[Theorem 3.6]{HW18} gives an even more precise information on $N'$ and $N$. It is easy to see that (B) implies (A).

One defines the formal periods\footnote{Warning: In other sources $P(M)$ often denotes the space of {\em actual periods}, i.e.\@ the image of $P(M)$ under $\eval_\int$. This is only isomorphic to $P(M)$ as defined here if the period conjecture holds.} $P(M)$ of $M \in \mathcal{A}$ as the image of $H_B(M) \otimes_\Q H_{dR}(M)^\vee$  in $P(\mathcal{A})$. 
The above theorem is related to a more concrete description of $P(M)$:

\begin{SATZ}[see Theorem~\ref{SATZQ2}]\label{SATZPERINTRO}
If $M \in \mathcal{A}$, then $P(M)$ is isomorphic to the vector space $H_B(M) \otimes_\Q H_{dR}(M)^\vee$ modulo the relations
\[ \sum_{i=1}^m \sigma_i  \otimes \omega_i  = 0 \] 
for every exact sequence
\[ \xymatrix{  0 \ar[r] & N' \ar[r] & M^m  \ar[r] & N  \ar[r] & 0    } \]
and for all $\sigma=(\sigma_i) \in H_B(N')$ and $\omega=(\omega_i) \in H_{dR}(N)^\vee$. 
\end{SATZ}

Note that one cannot simply take $m=1$ in most cases. Sometimes $m=2$ is sufficient, see below. 

Theorem~\ref{SATZPERINTRO} rises the question under which circumstances $P(M)$ can even be described only in terms of $M$ and its endomorphisms, i.e.\@ when is $P(M)$ equal to the vector space $H_{B}(M) \otimes_\Q H_{dR}(M)^\vee$
modulo the relations
\[  \alpha(\sigma) \otimes \omega =  \sigma \otimes \alpha^\vee(\omega)  \]
for all $\alpha \in \End(M)$? This is the case for {\em principal} objects (cf.\@ Definition~\ref{DEFWP}) in $\mathcal{A}$, in particular for all semi-simple objects (e.g.\@ pure motives):

\begin{PROP}[see Proposition~\ref{PROPWP}]
There is a surjection
\[  H_B(M) \otimes_{\End(M)^{\op}} H_{dR}(M)^\vee \twoheadrightarrow P(M). \]
which is an isomorphism if $M$ is principal. In this case one can take $m=2$ in Theorem \ref{SATZPERINTRO}.
\end{PROP}
As an application we show in Section~\ref{SECT1MOT} that the structure theorems for periods of 1-motives obtained in \cite[Chapter 9]{HW18} can be formally  explained by the fact that certain sufficiently ``saturated'' 1-motives are principal.

\begin{PAR}
Assume that $\mathcal{A} \cong \cat{$G$-Rep}$ is in addition neutral Tannakian, where $G$ is
the corresponding motivic Galois group defined over $\Q$, and that the fiber functors and the comparison isomorphism $\int$ are compatible with the tensor structure. 
Denote by $X$ the $G_\Qbar$-torsor of isomorphisms between the fiber functors $H_{B} \otimes_\Q \Qbar$ and $H_{dR}$. The $G_\Qbar$-torsor $X$ is non-canonically isomorphic
to $G_\Qbar$ and thus its coordinate ring $\OOO_X$ is non-canonically isomorphic to $\OOO_G \otimes_\Q \Qbar$.

By construction (cf.\@ Deligne and Milne \cite[Chapter II]{LNM900}) the space of formal periods as defined above is equal to the coordinate ring of $X$: 
\[ P(\mathcal{A}) =\OOO_X. \]
The ring structure depends on the tensor product and is on $P(\mathcal{A})$ given by 
\[ (\sigma_1\otimes \omega_1)_{M_1} \cdot (\sigma_2\otimes \omega_2)_{M_2} = ((\sigma_1 \otimes \sigma_2) \otimes (\omega_1 \otimes \omega_2))_{M_1 \otimes M_2}.   \]

Elements of (\ref{eqBdR}) can be seen as functions on $X$ as follows: A natural isomorphism $\tau$ between fiber functors is given by isomorphisms
\[ \tau_M: H_{B}(M) \otimes_\Q \Qbar \rightarrow H_{dR}(M)   \]
for all $M$ compatible with morphisms in $\mathcal{A}$. An element $\gamma:=(\sigma\otimes \omega)_M$ in $\OOO_X$ associates with such a natural isomorphism
the value 
\[ \eval_\tau(\gamma):= \omega(\tau_M(\sigma)). \] 
For the definition of $\eval_\tau$ it is irrelevant whether $\tau$ is compatible with the tensor structure or not. 
If it is, then $\eval_\tau: \OOO_X \rightarrow \Qbar$ is a ring homomorphism. In this setting the period conjecture is also equivalent to statement:
\begin{itemize}
\item[(C)] $X$ is connected and $\int$, considered as a point in  $X(\C)$, is generic. 
\end{itemize}
This article, however, is written entirely in the Abelian setting, the tensor structure will not play any role whatsoever.  
\end{PAR}

The author thanks Annette Huber and Nicola Nesa for raising questions that resulted in writing this article, for interesting discussions, and useful comments.

\section{Formal periods on Abelian categories}

As explained in the introduction, the tensor structure is irrelevant for the structure of formal periods as vector space. To obtain the results we work thus entirely in the following setting: 
\begin{DEF}\label{DEFFIBFUN}
Let $\mathcal{A}$ be a $\Q$-linear Abelian category and let $K \supset \Q$ be a field. A {\bf fiber functor}, defined over $K$, is a $\Q$-linear, exact, and faithful functor
\[ \mathcal{A} \rightarrow \cat{ f.d.-$K$-Vect} \]
into the category of finite dimensional $K$-vector spaces.
\end{DEF}

\begin{PAR}\label{PARTUPLE}
Fix a subfield $K$ of $\C$. We consider a quadruple $(\mathcal{A}, H_{B}, H_{dR}, \int)$ of a  $\Q$-linear Abelian category $\mathcal{A}$ together with two fiber functors $H_{B}, H_{dR}$, defined over $\Q$, and $K$, respectively, and a natural isomorphism
\[ \int: H_{B} \otimes_\Q \C \rightarrow H_{dR} \otimes_K \C. \]
\end{PAR}

\begin{DEF}[{cf.\@ \cite[Chapter 13]{HMS17}}]\label{DEFFORMALPERIODS}
We define the space of {\bf formal periods} $P(\mathcal{A})$ as the $K$-vector space
\[ \bigoplus_M H_B(M) \otimes_\Q H_{dR}(M)^\vee   \]
subject to the relations
\begin{equation}\label{eqprincipalrel}
 (\sigma \otimes \alpha^\vee(\omega))_M = (\alpha(\sigma) \otimes \omega)_N 
\end{equation}
for all morphisms $\alpha: M \rightarrow N$ in $\mathcal{A}$,  $\sigma \in H_B(M)$, and $\omega \in H_{dR}(N)^\vee$.

We define $P(M)$ as the image of $H_{B}(M) \otimes_\Q H_{dR}(M)^\vee$ in $P(\mathcal{A})$. 
\end{DEF}

\begin{BEM}
As discussed in the Introduction, if $\mathcal{A}$ is neutral Tannakian, and if $H_B$ and $H_{dR}$ are compatible with the tensor structure then $P(\mathcal{A})$ is precisely
the coordinate ring of the torsor $X$ of isomorphisms between $H_B \otimes_\Q K$ and $H_{dR}$. 
\end{BEM}

\begin{DEF}We say that the quadruple $(\mathcal{A}, H_{B}, H_{dR}, \int)$ as in \ref{PARTUPLE} satisfies the {\bf period conjecture} if the
$K$-linear map
\begin{eqnarray*}
 \eval_\int: P(\mathcal{A}) &\rightarrow& \C \\
 (\sigma \otimes \omega)_M &\mapsto&  \int_{\sigma} \omega := \omega (\int_M(\sigma))
 \end{eqnarray*} 
is {\em injective}. 
\end{DEF}

\begin{LEMMA}\label{LEMMASUM}For $i=1 \dots n$, let $M_i \in \mathcal{A}$, $\sigma_i \in H_B(M_i)$, and $\omega_i \in H_{dR}(M_i)^\vee$. 
Then we have in $P(\mathcal{A})$ the equality
\[ \sum_{i=1}^n (\sigma_i \otimes \omega_i)_{M_i} =   \left(\sum_{i=1}^n  (\sigma_i \otimes  \omega_i) \right)_{\bigoplus_i M_i}.  \]
\end{LEMMA}
\begin{proof}
It suffices to see this for $n=2$. The inclusion $M_1 \hookrightarrow M_1 \oplus M_2$ (or equally well the projection $M_1 \oplus M_2 \rightarrow M_1$) gives 
\[ ((\sigma_1,0) \otimes (\omega_1,0))_{M_1 \oplus M_2}  =  (\sigma_1 \otimes \omega_1)_{M_1} \] 
and similarly $M_2 \hookrightarrow M_1 \oplus M_2$ gives
\[ ((0, \sigma_2) \otimes (0,\omega_2))_{M_1 \oplus M_2}  =  (\sigma_2 \otimes \omega_2)_{M_2}. \]
The statement follows.  
\end{proof}

\section{Formal periods of a single object}

Let the setting be as in \ref{PARTUPLE}.
We now redefine formal periods in a different way (Theorem~\ref{SATZQ2} below states that indeed $P^\infty(M) = P(M)$):

\begin{DEF}\label{DEFFPDEPTH}
Let $(\mathcal{A}, H_{B}, H_{dR})$ be as in \ref{PARTUPLE} and
let $M \in \mathcal{A}$. 
The space of {\bf formal periods of depth $k$} of $M$, denoted $P^k(M)$, is the quotient of the vector space $H_{B}(M) \otimes_\Q H_{dR}(M)^\vee$ modulo the subspace generated by 
\[ \sum_{i=1}^m \sigma_i  \otimes \omega_i   \] 
for every exact sequence
\[ \xymatrix{  0 \ar[r] & N' \ar[r] & M^m  \ar[r] & N  \ar[r] & 0,    } \]
$(\sigma_i) \in H_B(N')$, and $(\omega_i) \in H_{dR}(N)^\vee$, such that $m \le k$. 

We also define 
\[ P^\infty(M) := \lim_k P^k(M). \] 
\end{DEF}
Since the $P^k(M)$ are all (successive) quotients of the same finite dimensional vector space, we have obviously  $P^k(M) = P^{k+1}(M) = \cdots =  P^\infty(M)$ for some $k$. 

\begin{KEYLEMMA}\label{LEMMAP}
For each relation of the form $\sum_{i=1}^m \sigma_i \otimes \omega_i = 0$ in $P^\infty(M)$ there is an exact sequence
\[ \xymatrix{  0 \ar[r] & N' \ar[r] & M^m  \ar[r] & N  \ar[r] & 0    } \]
with $(\sigma_i) \in H_B(N')$ and $(\omega_i) \in H_{dR}(N)^{\vee}$.
\end{KEYLEMMA}
Notice that this is not a tautology because a priori $P^\infty(M)$ is defined as the quotient modulo the vector subspace {\em generated} by such expressions.  
\begin{proof}
The tensor product $H_B(M) \otimes_\Q H_{dR}(M)^\vee$ can be defined as the quotient of the Abelian monoid generated by abstract symbols
\[ \sigma \otimes \omega \]
with $\sigma \in  H_B(M)$ and $\omega \in H_{dR}(M)$
 modulo the relations
\begin{equation*}
\begin{array}{rrcl}
\text{(i)} & x + \sigma  \otimes 0 \,\ \sim \,\ x + 0  \otimes \omega & \sim & x  \\
\text{(ii)} & x + \sigma_1  \otimes  \omega + \sigma_2 \otimes  \omega & \sim & x + (\sigma_1 + \sigma_2) \otimes  \omega \\
\text{(iii)} & x + \sigma  \otimes  \omega_1 + \sigma \otimes  \omega_2& \sim & x + \sigma \otimes  (\omega_1 + \omega_2) \\
\text{(iv)} & x + \sigma \otimes (\lambda \omega) & \sim  & x + (\lambda \sigma) \otimes  \omega
\end{array}
\end{equation*}
for all formal sums $x = \sum_i \sigma_i' \otimes \omega_i'$, and for all $\sigma, \sigma_1, \sigma_2 \in H_B(M)$,  $\omega, \omega_1, \omega_2 \in H_{dR}(M)^\vee$,  $\lambda \in \Q^*$.

Let $R$ be the submonoid of all formal expressions $\sum_i^n{} \sigma_i \otimes \omega_i$ for each exact sequence
\[ \xymatrix{  0 \ar[r] & N' \ar[r]& M^n  \ar[r] & N  \ar[r] & 0    } \]
and $\sigma=(\sigma_i) \in H_{B}(N')$ and $\omega=(\omega_i) \in H_{dR}(N)^\vee$. This  is indeed a submonoid with 0 taking into account the direct sum of exact sequences
\[ \xymatrix{  0 \ar[r] & N'_1 \oplus N_2' \ar[r] & M^{n_1} \oplus M^{n_2}  \ar[r] & N_1 \oplus N_2  \ar[r] & 0    } \]
and neglecting the ordering is legitimate\footnote{i.e.\@ a priori we have $\sum_i^n{} \sigma_i \otimes \omega_i \in R$ if and only if $\sum_i^n{} \sigma_{\tau(i)} \otimes \omega_{\tau(i)} \in R$ for any permutation $\tau$} because the symmetric group $S_n$ acts on $M^n$. 

We will show below the following properties:
\begin{equation*}
\begin{array}{rrcl}
\text{(i)} &x + \sigma \otimes 0 \in R \,\  \Leftrightarrow \,\ x + 0  \otimes \omega  \in R  & \Leftrightarrow & x \in R \\
\text{(ii)} & x + \sigma_1  \otimes  \omega + \sigma_2 \otimes  \omega \in R & \Leftrightarrow & x + (\sigma_1 + \sigma_2) \otimes  \omega \in R \\
\text{(iii)} & x + \sigma  \otimes  \omega_1 + \sigma \otimes  \omega_2 \in R & \Leftrightarrow & x + \sigma \otimes  (\omega_1 + \omega_2) \in R \\
\text{(iv)} & x + \sigma \otimes (\lambda \omega) \in R& \Leftrightarrow& x + (\lambda \sigma) \otimes  \omega \in R  \\
\text{(v)} & \sum \sigma_i \otimes \omega_i \in R & \Rightarrow &  \sum \sigma_i \otimes (\mu \omega_i) \in R   \\
\text{(vi)} & \sigma  \otimes \omega + (-\sigma) \otimes  \omega \in R & 
\end{array}
\end{equation*}
for all $x \in R, \sigma, \sigma_1, \sigma_2 \in H_B(M)$,  $\omega, \omega_1, \omega_2 \in H_{dR}(M)^\vee$,  $\lambda \in \Q^*$ and $\mu \in K$.

Let now $\gamma = \sum_{i=1}^n \sigma_i \otimes \omega_i$ be a formal sum which maps to zero in $P^\infty(M)$. This means using (v) and the sum in $R$ that there is an element
\[  \gamma' = \sum_{i=1}^m \sigma_i' \otimes \omega_i'   \]
in $R$ mapping to $\gamma$ in the tensor product. Properties (i--iv) together with the observation before show that also  $\gamma \in R$. 
This gives the Lemma, provided that we prove the properties (i--v). 

Recall that $H_B$ and $H_{dR}^\vee$ are both exact functors by definition. 

For ``$x \in R \Rightarrow x + 0  \otimes \omega  \in R$''  the exact sequence
\[ \xymatrix{  0 \ar[r] & 0 \ar[r] & M  \ar[r] & M  \ar[r] & 0    } \]
shows that for any $\omega \in H_{dR}(M)^\vee$ we have $0 \otimes \omega \in R$.
For ``$x + 0  \otimes \omega \in R \Rightarrow x \in R$'' consider an exact sequence 
\[ \xymatrix{  0 \ar[r] & N' \ar[r] & M^n \oplus M  \ar[r] & N  \ar[r] & 0    } \]
with $(\sigma_1', \dots, \sigma_n', 0) \in H_B(N')$ and $(\omega_1', \dots, \omega_n',\omega) \in H_{dR}(N)^\vee$.
We get a morphism of exact sequences
\[ \xymatrix{  0 \ar[r]  & N' \cap M^n   \ar[r] \ar[d]^{} & M^n   \ar[d]^{\iota}  \ar[r] & L \ar[d]  \ar[r] & 0    \\
  0 \ar[r] & N' \ar[r] & M^n \oplus M  \ar[r] & N  \ar[r] & 0    } \]
where the morphism $\iota: M^n \rightarrow M^n \oplus M$ is the obvious inclusion and $L$ is the quotient and obviously  $(\sigma_1', \dots, \sigma_n') \in H_B(N' \cap M^n)$ and $(\omega_1', \dots, \omega_n') \in H_{dR}(L)^\vee$. The other assertions of (i) are dual. 

``(ii) $\Leftarrow$'' follows from (vi) and ``(ii) $\Rightarrow$''.
For ``(ii) $\Rightarrow$''  consider an exact sequence
\[ \xymatrix{  0 \ar[r] & L' \ar[r] & M^n \oplus M^2  \ar[r] & L  \ar[r] & 0    } \]
with $(\sigma_1', \dots, \sigma_n', \sigma_1,\sigma_2) \in H_B(L')$ and $(\omega_1', \dots, \omega_n',\omega,\omega) \in H_{dR}(L)^\vee$.
Consider also the morphism 
\[ p: M^n \oplus M^2 \rightarrow M^n \oplus M \] 
which is the identity on $M^n$ and the sum on $M^2$. We get a morphism of exact sequences
\[ \xymatrix{  0 \ar[r]  & L' \ar[r] \ar[d]^{} & M^n \oplus M^2  \ar[d]^{p}  \ar[r] & L\ar[d]  \ar[r] & 0    \\
  0 \ar[r] & N' \ar[r] & M^n \oplus M  \ar[r] & N  \ar[r] & 0    } \]
where $N':=p(L')$ and $N$ is the quotient. One checks that $(\sigma_1', \dots, \sigma_n', \sigma_1 + \sigma_2) \in H_B(N')$ and $(\omega_1', \dots, \omega_n',\omega) \in H_{dR}(N)^{\vee}$.

``(iii) $\Leftarrow$'' follows from (vi) and ``(iii) $\Rightarrow$'' and
the proof of ``(iii) $\Rightarrow$''  is dual to ``(ii) $\Rightarrow$''  proven above.

For (iv) consider an exact sequence
\[ \xymatrix{  0 \ar[r] & N' \ar[r] & M^n \oplus M  \ar[r] & N  \ar[r] & 0    } \]
with $(\sigma_1', \dots, \sigma_n', \sigma) \in H_B(N')$ and $(\omega_1', \dots, \omega_n', \lambda \omega) \in H_{dR}(N)^\vee$.
We may apply the automorphism given by $(1, \dots, 1, \lambda)$ to $M^n \oplus M$. It maps the above sequence to a shifted sequence
\[ \xymatrix{  0 \ar[r] & L' \ar[r] & M^n \oplus M  \ar[r] & L  \ar[r] & 0    } \]
and the above elements to $(\sigma_1', \dots, \sigma_n', \lambda \sigma) \in H_B(L')$ and $(\omega_1', \dots, \omega_n', \omega) \in H_{dR}(L)^\vee$.
The statement follows. 

Property (v) follows directly from the definition of $R$.

For (vi) consider the exact sequence
\[ \xymatrix{  0 \ar[r] & M  \ar[rr]^-{(\id, \id)} & & M^2  \ar[rr]^-{(\id, -\id)} & & M \ar[r] & 0,    } \]
$(\sigma,\sigma) \in H_B(M^2)$, and $(\omega,-\omega) \in H_{dR}(M^2)^\vee$. 
\end{proof}

\begin{BEM}
The Key Lemma obviously implies that $P^k(M) = P^{k+1}(M) = \cdots =  P^\infty(M)$ for $k=\dim(H_B(M))$.
\end{BEM}

\begin{PAR}
Recall Quillen's {\bf Q-construction} \cite{Qui73}. $Q\mathcal{A}$ is the category with the same objects as $\mathcal{A}$ and with morphisms being isomorphism classes of spans $f=(p_f, \iota_f)$ of the form
\[ \xymatrix{ & N \ar@{->>}[ld]_-{p_f} \ar@{^{(}->}[rd]^-{\iota_f} & \\ M & & M'.  } \]
Composition is given by the fiber product. We will need a variant of this: $\widetilde{Q}\mathcal{A}$ is defined in the same way but with morphisms being equivalence classes generated by $f \sim g$ if there exists a commutative diagram:
\[ \xymatrix{
& N \ar@{->>}[ld]_-{p_f} \ar@{^{(}->}[dd] \ar@{^{(}->}[rd]^-{\iota_f} &  \\
M &   & M' \\
& N ' \ar@{->>}[lu]^-{p_g} \ar@{^{(}->}[ru]_-{\iota_g}  & 
}\]
\end{PAR}

\begin{LEMMA}\label{LEMMAQ}
The category $\widetilde{Q}\mathcal{A}$ is filtered. 
\end{LEMMA}
\begin{proof}
For two objects $A, B$ in $Q\mathcal{A}$ there are obvious morphisms $A \rightarrow A \oplus B$ and $B \rightarrow A \oplus B$.

It therefore suffices to see that for any two morphisms 
\[ f: \vcenter{ \xymatrix{ & N \ar@{->>}[ld]_-{p_f} \ar@{^{(}->}[rd]^-{\iota_f} & \\ M & & M'  }  }
\quad \text{and} \quad g: \vcenter{  \xymatrix{ & N' \ar@{->>}[ld]_-{p_g} \ar@{^{(}->}[rd]^-{\iota_g} & \\ M & & M'  } } \]
 there is a morphism $h: M' \rightarrow M''$ such that $h \circ f = h \circ g$ in $\widetilde{Q}\mathcal{A}$.
We will show that the following morphism does the job: 
\[ h: \vcenter{  \xymatrix{ & M' \oplus N \oplus N' \ar@{->>}[ld]_-{(\id, \iota_f, \iota_g)} \ar@{=}[rd] & \\ M'  & & M' \oplus N \oplus N' .  } }\]
Consider the composition
\[ h \circ f: \vcenter{  \xymatrix@C=4em{ & & \Box \ar@{->>}[ld] \ar@{^{(}->}[rd]  & &   \\
& N \ar@{->>}[ld]_-{p_f} \ar@{^{(}->}[rd]_-{\iota_f} & & M' \oplus N \oplus N' \ar@{->>}[ld]^-{(\id, \iota_f, \iota_g)} \ar@{=}[rd]  \\ 
M & & M' & & M' \oplus N \oplus N' .  } } \]
and the diagram
\[ \xymatrix@C=4em{
& N \ar@{->>}[ld]_-{p_f} \ar@{^{(}->}[dd] \ar@{^{(}->}[rd]^-{(0,\id,0)} &  \\
M &   & M' \oplus N \oplus N' \\
& \Box \ar@{->>}[lu] \ar@{^{(}->}[ru]  & 
}\]
in which the morphism $N \hookrightarrow \Box$ is induced by the pair $(\id, (0, \id, 0))$.
This shows that $h \circ f$ is equivalent to the upper span in the following commutative diagram:
\[ \xymatrix@C=4em{
& N \ar@{->>}[ld]_{p_f} \ar@{^{(}->}[d] \ar@{^{(}->}[rd]^-{(0,\id,0)} \\
M & \ar@{->>}[l] N' \oplus N \ar@{^{(}->}[r]^-{(0,\id,\id)}    & M' \oplus N \oplus N'  \\
& N' \ar@{->>}[lu]^{p_g} \ar@{^{(}->}[ru]_-{(0,0,\id)} \ar@{^{(}->}[u] 
}\]
Similarly $h \circ g$ is equivalent to the lower span and thus $h \circ f = h \circ g$ in $\widetilde{Q} \mathcal{A}$. 
\end{proof}

\begin{PROP}\label{PROPQ}
For all $1 \le k \le \infty$, the association $M \mapsto P^k(M)$ of Definition~\ref{DEFFPDEPTH} defines a functor
\[ P^k: \widetilde{Q}\mathcal{A} \rightarrow \cat{ f.d.-$K$-Vect}. \]
The functor $P^\infty$ maps all morphisms to injections (the other $P^k$ do not in general). 

For all $f: M \rightarrow M'$ in $\widetilde{Q}\mathcal{A}$ the diagram
\[ \xymatrix{ P^\infty(M) \ar@{^{(}->}[d]_{P^\infty(f)} \ar[r]^-{\eval_\int} &  \C \\
P^\infty(M') \ar[ru]_-{\eval_\int}
} \]
commutes. 
\end{PROP}
In other words, the  filtered system of vector spaces and inclusions $M \mapsto P^\infty(M)$ mimics abstractly the system of subspaces of periods and their inclusions. 
\begin{proof}
For a subobject $N' \hookrightarrow M$ with quotient $N$ we have a morphism
\[ H_B(N') \otimes_\Q H_{dR}(N')^{\vee} \cong H_B(N') \otimes_\Q \frac{H_{dR}(M)^\vee}{H_{dR}(N)^\vee}  \hookrightarrow \frac{H_B(M) \otimes_\Q H_{dR}(M)^\vee}{H_B(N') \otimes_\Q H_{dR}(N)^\vee}  \]
and $P^k(M)$ is a quotient of the target because $k \ge 1$ by assumption.
Therefore this defines a morphism $H_B(N') \otimes_\Q H_{dR}(N')^{\vee} \rightarrow P^k(M)$.

For a quotient $M \twoheadrightarrow N$ with kernel $N'$  we have a morphism 
\[ H_B(N) \otimes_\Q H_{dR}(N)^{\vee} \cong \frac{H_B(M)}{H_B(N')} \otimes_\Q H_{dR}(N)^{\vee}  \hookrightarrow \frac{H_B(M) \otimes_\Q H_{dR}(M)^\vee}{H_B(N') \otimes_\Q H_{dR}(N)^\vee}.   \]
and $P^k(M)$ is again a quotient of the target.
Therefore this defines a morphism $H_B(N) \otimes_\Q H_{dR}(N)^{\vee} \rightarrow P^k(M)$.

We have to verify that the relations are respected. For the case of a subobject consider a relation $\sum_i \sigma_i \otimes \omega_i$ given by a sequence
\[ \xymatrix{  0 \ar[r] & K' \ar[r] & N^m  \ar[r] & K  \ar[r] & 0    } \]
with $(\sigma_i) \in H_B(K')$ and $(\omega_i) \in H_{dR}(K)^\vee$ and $m \le k$.
Consider the monomorphism of exact sequences
\[ \xymatrix{  0 \ar[r] & K' \ar@{=}[d] \ar[r] & N^m \ar@{^{(}->}[d] \ar[r] & K \ar@{^{(}->}[d] \ar[r] & 0    \\
0 \ar[r] & K' \ar[r] & M^m  \ar[r] & \lefthalfcap  \ar[r] & 0    } \]
where $\lefthalfcap$ denotes the push-out.
The relation is mapped to $\sum_i \sigma_i \otimes \widetilde{\omega}_i$, where $\widetilde{\omega}_i$ is any lift of $\omega_i$ to $M$. 
Then $(\widetilde{\omega}_i)$ is in $H_{dR}(\lefthalfcap)^\vee$ thus the new relation
is valid in $P^k(M)$.
The case of a quotient is similar. An easy verification shows that the functor is compatible with composition and that it factors through $\widetilde{Q}\mathcal{A}$. 

Injectivity: Consider first a morphism $P^\infty(N') \rightarrow P^\infty(M)$ given by a subobject $N' \hookrightarrow M$ and let $\gamma = \sum_i \sigma_i \otimes \omega_i$ 
be a tensor that is mapped to zero. Represent the image as $\sum_i \sigma_i \otimes \widetilde{\omega}_i$ with $\widetilde{\omega_i }\in H_{dR}(M)^\vee$ (lift of $\omega_i \in H_{dR}(N')^\vee$). The fact that $\gamma$ maps to zero means (using Key Lemma~\ref{LEMMAP} --- this is the only place where we must have $P^\infty$ instead of $P^k$) that there is a sequence 
\[ \xymatrix{  0 \ar[r] & L' \ar[r] & M^n  \ar[r] & L  \ar[r] & 0    } \]
with $(\sigma_i) \in H_B(L')$ and $(\widetilde{\omega}_i) \in H_{dR}(L)^\vee$. Consider the monomorphism of exact sequences
\[ \xymatrix{ 
 0 \ar[r] & L' \cap (N')^n \ar@{^{(}->}[d]  \ar[r] & (N')^n  \ar@{^{(}->}[d]   \ar[r] & L'' \ar@{^{(}->}[d]   \ar[r] & 0   \\
 0 \ar[r] & L' \ar[r] & M^n \ar[r] & L  \ar[r] & 0  
   } \]
   where $L''$ is the image of $(N')^n$ in $L$. 
We see that $\sigma_i \in H_B(L' \cap (N')^n)$ and $\omega_i \in H_{dR}(L'')^\vee$, hence $\gamma$ is already 0 in $P^\infty(N')$. 
The case of a quotient is dual. 
\end{proof}

\section{Comparison}

Let the setting be as in \ref{PARTUPLE}.
Recall the spaces of formal periods $P(\mathcal{A})$ and $P(M)$ from Definition~$\ref{DEFFORMALPERIODS}$  as well as $P^k(M)$ and $P^\infty(M)$ from Definition~\ref{DEFFPDEPTH}. 

\begin{SATZ}\label{SATZQ2}
We have 
\[  P(\mathcal{A}) = \colim_{\widetilde{Q}\mathcal{A}} P^\infty = \colim_{\widetilde{Q}\mathcal{A}} P^k  \]
for all $k \ge 1$. 
Furthermore, for each $M \in \mathcal{A}$, the space $P^\infty(M)$ (but not necessarily the $P^k(M)$) {\em injects} into $P(\mathcal{A})$ and thus
\[ P^\infty(M) \cong P(M).   \]
\end{SATZ}
\begin{proof}
All three $K$-vector spaces are in a natural way quotients of $\bigoplus_M H_B(M) \otimes_\Q H_{dR}(M)^\vee$. We have to see that the relations correspond. 
From left to right, consider a morphism $\alpha: N \rightarrow M$ in $\mathcal{A}$ and
factor
\[ \xymatrix{ N \ar@{->>}[r]^-p & \alpha(N)  \ar@{^{(}->}[r]^-\iota & M } \]
A tensor $\sigma \otimes \omega \in H_B(N) \otimes_\Q H_{dR}(M)^\vee$ gives an element $p(\sigma) \otimes \iota^\vee(\omega) \in H_B(\alpha(N)) \otimes_\Q H_{dR}(\alpha(N))^\vee$. 
The morphism in $\widetilde{Q} \mathcal{A}$ corresponding to $p$ maps $(p(\sigma) \otimes \iota^\vee(\omega))_{\alpha(N)}$ to $(\sigma \otimes \alpha^{\vee}(\omega))_N$ and the map
in $\widetilde{Q} \mathcal{A}$ corresponding to $\iota$ (going in the other direction) maps $(p(\sigma) \otimes \iota^\vee(\omega))_{\alpha(N)}$ to $(\alpha(\sigma) \otimes \omega)_M$.
We thus get the same relation in $\colim_{\widetilde{Q}\mathcal{A}} P^\infty$ and also in $\colim_{\widetilde{Q}\mathcal{A}} P^k$ for $k \ge 1$.

In the other direction, for each morphism $f: M \rightarrow M'$ in $Q\mathcal{A}$  we get an equality $(\sigma\otimes \omega)_M = P^k(f)(\sigma \otimes \omega)_{M'}$ in the colimit. 
We have to find a corresponding relation on the left. It suffices to see this for morphisms in $Q\mathcal{A}$ corresponding to mono- or epimorphisms. For those, 
the relation (\ref{eqprincipalrel}) for this particular mono- or epimorphism does the job. For each relation $\sum_i \sigma_i \otimes \omega_i = 0$ in $P^\infty(M)$ corresponding to an exact sequence
\[ \xymatrix{ 0 \ar[r] & N'  \ar[r]^\iota  & M^n \ar[r]^p & N \ar[r] & 0 } \]
and $(\sigma_i) \in H_B(N')$ and $(\omega_i) \in H_{dR}(N)^\vee$,
we have in $P(\mathcal{A})$ the relation
\[  (\iota(\sigma) \otimes p^\vee(\omega))_{M^n} =  (p \iota(\sigma) \otimes \omega)_{N} =0. \]
Hence Lemma~\ref{LEMMASUM} shows $\sum_i (\sigma_i \otimes \omega_i)_M = 0$ in $P(\mathcal{A})$. 

The canonical morphism $P^\infty(M) \rightarrow \colim_{\widetilde{Q}\mathcal{A}} P^\infty$ is injective because $\widetilde{Q}\mathcal{A}$ is filtered (see Lemma~\ref{LEMMAQ}) and all morphisms are mapped to injections by Proposition~\ref{PROPQ} (which holds only for $P^\infty$, and not for the $P^k$). 
\end{proof}

\begin{KOR}
\label{KORAB}
The period conjecture for $(\mathcal{A}, H_{B}, H_{dR}, \int)$ is equivalent to the following implication: 
If there are $M \in\mathcal{A}$, $\sigma \in H_B(M)$, and  $\omega \in H_{dR}(M)^\vee$ such that 
\[  \int_\sigma \omega  = 0 \]
then there is an exact sequence
 \[ \xymatrix{ 0 \ar[r] & N' \ar[r] & M \ar[r] & N \ar[r]  & 0 }  \]
 with $\sigma \in H_B(N')$ and $\omega \in H_{dR}(N)^\vee$.
\end{KOR}
\begin{proof}
Let 
\[ \sum \int_{\sigma_i} \omega_i = 0 \]
 with $\sigma_i \in H_B(M_i)$ and $\omega_i \in H_{dR}(M_i)^\vee$ be a period relation. This might be seen as 
$\int_\sigma \omega = 0$
for $\sigma:= \bigoplus_i \sigma_i$ and $\omega:= \bigoplus_i \omega_i$ on $M = \bigoplus M_i$. If the implication holds, then
there is an exact sequence
 \[ \xymatrix{ 0 \ar[r] & N' \ar[r]^s & M \ar[r]^p & N \ar[r]  & 0 }  \]
 such that $\sigma \in H_B(N')$ and $\omega \in H_{dR}(N)^\vee$. This implies using Lemma~\ref{LEMMASUM}
 \[ 0 = (ps(\sigma) \otimes  \omega)_N = (s(\sigma) \otimes  p^\vee(\omega))_M = \sum_i (\sigma_i \otimes \omega_i)_{M_i}  \]
 in $P(\mathcal{A})$. This shows that $\eval_\int$ is injective. 

Now let $M \in \mathcal{A}$, $\sigma \in H_B(M)$ and $\omega \in H_{dR}(M)^\vee$ be such that $\int_{\sigma} \omega = 0$. If the period conjecture holds, the formal period $(\sigma \otimes \omega)_M$ is zero. Therefore, since $P(M)=P^{\infty}(M)$ by Theorem~\ref{SATZQ2}, we have
 an exact sequence of the required form by the Key Lemma~\ref{LEMMAP}. Hence the implication holds. 
\end{proof}

\begin{BEM}
The above reformulation of the period conjecture is motivated by the Theorem of Huber and W\"ustholz
\cite[Theorem 3.6]{HW18}. In \cite{HW18} the period conjecture for 1-motives (cf.\@ Theorem~\ref{SATZHW} in the Introduction) is deduced from the statement in the reformulation --- the easy direction of the above Corollary.  One of the motivations of this article was to investigate whether the statement is actually always equivalent to the period conjecture. 
\end{BEM}

\section{Basic properties of formal periods}

Let the setting be as in \ref{PARTUPLE}.
The following Proposition subsumes some elementary properties of formal periods in the Abelian setting. 

\begin{PROP}\label{PROPBASICPROP}
\begin{enumerate}
\item If $N \subset M$ is a subobject or if $M \twoheadrightarrow N$ is a quotient then we have

$P(M \oplus N) = P(M).$
\item 
$P(M^n) = P(M)$ for every $n \ge 1$.
\item If $M_0 \in \mathcal{A}_0$ and $M_1 \in \mathcal{A}_1$ lie in Abelian subcategories of $\mathcal{A}$ such that $\Hom(\mathcal{A}_0, \mathcal{A}_1) = \Hom(\mathcal{A}_1, \mathcal{A}_0) = 0$
then

$P(M_0 \oplus M_1) = P(M_0) \oplus P(M_1).$
\item For two $\Q$-vector spaces $X \not= 0$ and $L$ and an exact sequence
\[ \xymatrix{ 0 \ar[r] & M_0 \otimes_\Q X \ar[r] &  M \ar[r] &  M_1 \otimes_\Q L \ar[r] & 0  }\]
we have
\[ P(M) = P(\widetilde{M}) \]
where 
\[ \xymatrix{ 0 \ar[r] & M_0 \ar[r] &  \widetilde{M} \ar[r] &  M_1 \otimes_\Q \Hom(X, L) \ar[r] & 0  }\]
is constructed as the pushout of $M \otimes_\Q X^\vee$ along $\tr: \Hom(X, X) \rightarrow \Q$.
\end{enumerate}
\end{PROP}
\begin{proof}
1.\@ is similar to Lemma~\ref{LEMMASUM} and left to the reader.

2.\@ follows by induction from 1.\@ (or from Lemma~\ref{LEMMASUM}).

3.\@ follows directly from the definition of $P^\infty(M)$ and Theorem~\ref{SATZQ2}. 

4.\@ We can reconstruct the given sequence as the pull-back of 
\[ \xymatrix{ 0 \ar[r] & M_0 \otimes_\Q X \ar[r] &  \widetilde{M} \otimes_\Q X \ar[r] &  M_1 \otimes_\Q \Hom(X, L) \otimes_\Q X \ar[r] & 0  }\]
along 
\[ \frac{1}{d}\id_X \otimes -: L \mapsto \End(X) \otimes_\Q L = \Hom(X, L) \otimes_\Q X \]
where $d= \dim(X)$.

We have an injection $\iota_2: M \hookrightarrow \widetilde{M} \otimes_\Q X$ and a projection
$p_1: M \otimes_{\Q} X^\vee \rightarrow \widetilde{M}$.
Furthermore we have an obvious injection $\frac{1}{d}\id: M \rightarrow M \otimes_\Q X^\vee \otimes_\Q X$ and projection $\tr: \widetilde{M} \otimes_\Q X \otimes_\Q X^\vee \rightarrow \widetilde{M}$. 
This gives the following two diagrams in $\widetilde{Q} \mathcal{A}$:
\[ \xymatrix{ M \ar[r]^{\iota_2} \ar[rd]_-{\tr'}  & \widetilde{M} \otimes_\Q X \ar[d]^{(p_1 \otimes X)'} & \ar[l] \widetilde{M} \ar[d]^{p_1'}    \\
&  M \otimes_\Q X^\vee \otimes_\Q X  & \ar[l] M \otimes_\Q X^\vee   } \quad
\xymatrix{ \widetilde{M} \ar[r]^{p_1'} \ar[rd]_-{\frac{1}{d}\id}  & M \otimes_\Q X^\vee \ar[d]^{\iota_2 \otimes X^\vee} & \ar[l]  M \ar[d]^{\iota_2}     \\
&  \widetilde{M} \otimes_\Q X \otimes_\Q X^\vee   & \ar[l] \widetilde{M} \otimes_\Q X  } \]
where the unnamed morphisms are induced by any compatible choice of projection (resp.\@ injection).
Assume for the moment that the above diagrams are commutative in $\widetilde{Q} \mathcal{A}$. 
$P$ maps $\tr'$ and $\frac{1}{d}\id$ and also all unnamed horizontal morphisms to isomorphisms by 2. By the 2-out-of-6 property of isomorphisms therefore $P$ maps $\iota_2$ and $p_1'$ to isomorphisms. 

It remains to verify the commutativity of the diagrams.
The composition $(\iota_2 \otimes X^\vee) \circ p_1'$ is equal to $\tr'$ by means of the following equivalence of morphisms in $\widetilde{Q} \mathcal{A}$:
\[ \xymatrix{  & M \otimes_\Q X^\vee  \ar@{->>}[ld]_{p_1} \ar@{^{(}->}[rd]^{\iota_2 \otimes X^\vee} \ar@{^{(}->}[dd]^{\iota_2 \otimes X^\vee} \\
\widetilde{M} & &  \widetilde{M} \otimes_\Q X \otimes_\Q X^\vee   \\
&  \widetilde{M} \otimes_\Q X \otimes_\Q X^\vee \ar@{=}[ru] \ar@{->>}[lu]^-{\tr}  }\]
The composition $(p_1 \otimes X^\vee)' \circ \iota_2$ is given by
\[ \xymatrix@C=4em@R=3em{   & & \Box \ar@{->>}[ld] \ar@{^{(}->}[rd] \\
& M   \ar@{=}[ld] \ar@{^{(}->}[rd]^{\iota_2}  && M \otimes_\Q X^\vee \otimes_\Q X \ar@{->>}[ld]_{p_1 \otimes X} \ar@{=}[rd]  \\
M & &  \widetilde{M} \otimes_\Q X  & & M \otimes_\Q X \otimes_\Q X^\vee  } \]
where $\Box$ denotes the fiber product and there is an equivalence of morphisms in  $\widetilde{Q} \mathcal{A}$:
\[ \xymatrix@C=6em{  & \Box  \ar@{->>}[ld] \ar@{^{(}->}[rd] \ar@{<-^{)}}[dd]_{(\id, \frac{1}{d}\id)} \\
M & &  M \otimes_\Q X \otimes_\Q X^\vee   \\
&  M \ar@{^{(}->}[ru]_{\frac{1}{d}\id} \ar@{=}[lu]  }\]

\end{proof}

\section{The principal case}

Note that $\End(M)$ acts on $H_B(M)$ from the left and on $H_{dR}(M)^\vee$ from the right. 
There is always a surjection 
\[  H_B(M) \otimes_{\End(M)^{\op}} H_{dR}(M)^\vee \twoheadrightarrow P(M). \]
 In this section we investigate under which circumstances this is even an isomorphism. 
Indeed this will be the case for {\em principal} objects (cf.\@ Proposition~\ref{PROPWP}):

\begin{DEF}\label{DEFWP}
An object $M$ in an Abelian category is called {\bf left principal} if for every $m \ge 0$ and every subobject $N \subset M^m$ there is a morphism
\[ \xymatrix{ M^k \ar[r] & M^m    } \]
for some $k \ge 0$ with image $N$. 
There is an obvious dual notion of {\bf right principal}. 

We call $M$ {\bf principal} if the following equivalent conditions hold:
\begin{enumerate} 
\item[(P1 left)]  For each subobject $N' \subset M^m$ there is a diagram
\begin{equation}\label{diawe1} \vcenter{ \xymatrix{ 
\ker(\beta) \ar@{^{(}->}[d] \ar[rd]^{\alpha'} & \\ 
M^k \ar@{->}[d]^{\beta} \ar[r]^{\alpha} & M^m  \\
M^l } } \end{equation}
such that $N'=\im(\alpha')$.
\item[(P2 left)] Every subobject $N' \subset M^n$ lies in the smallest class $\mathcal{C}$ of subobjects $N' \subset M^m$ (for varying $m$) such that 
\begin{enumerate}
\item $(0 \subset M^m) \in  \mathcal{C}$ and $(M^m \subset M^m) \in  \mathcal{C}$ for all $m$;
\item If $(N' \subset M^m) \in \mathcal{C}$ and $\varphi: M^m \rightarrow M^k$ is a morphism then $(\varphi(N') \subset M^k) \in \mathcal{C}$;
\item If $(N' \subset M^k) \in \mathcal{C}$ and $\varphi: M^m \rightarrow M^k$ is a morphism then $(\varphi^{-1}(N') \subset M^m) \in \mathcal{C}$.
\end{enumerate} 
\item[(P1 right)] For every quotient $M^m \twoheadrightarrow N$ there is a diagram
\begin{equation}\label{diawe2}  \vcenter{ \xymatrix{ 
\coker(\beta) \ar@{<<-}[d]  \ar@{<-}[rd]^{\alpha'} & \\ 
M^k \ar@{<-}[d]^{\beta} \ar@{<-}[r]^{\alpha} & M^m  \\
M^l } } \end{equation}
with $N=\mathrm{coim}(\alpha')$
\item[(P2 right)] Every quotient $M^n \twoheadrightarrow N$ lies in the smallest class $\mathcal{C}$ of quotients $M^m \twoheadrightarrow N$ (for varying $m$) such that 
\begin{enumerate}
\item $(M^m \twoheadrightarrow 0) \in  \mathcal{C}$ and $(M^m \twoheadrightarrow M^m) \in  \mathcal{C}$ for all $m$;
\item If $(M^k \twoheadrightarrow N) \in \mathcal{C}$ and $\varphi: M^m \rightarrow M^k$ is a morphism then $(M^m \twoheadrightarrow \varphi^{-1}(N)) \in \mathcal{C}$;
\item If $(M^m \twoheadrightarrow N) \in \mathcal{C}$ and $\varphi: M^m \rightarrow M^k$ is a morphism then $(M^k \twoheadrightarrow \varphi(N)) \in \mathcal{C}$.
\end{enumerate} 
\end{enumerate}
\end{DEF}

In Axiom (P2 right) $M^k \twoheadrightarrow \varphi(N)$ denotes the pushout
\[ \xymatrix{  M^m \ar[r]^\varphi \ar@{->>}[d] &  M^k \ar@{->>}[d]  \\
N \ar[r] &  \varphi(N) } \]
and $M^m \twoheadrightarrow \varphi^{-1}(N)$ is induced by the factorization into epimorphism followed by monomorphism:
\[ \xymatrix{  M^m \ar[r]^\varphi \ar@{->>}[d] &  M^k \ar@{->>}[d]  \\
\varphi^{-1}(N)  \ar@{^{(}->}[r] &  N } \]
These are dual to the corresponding notions in (P2 left) with a less common notation. 

\begin{proof}[Proof of the equivalences in the definition.]
(P1 left) $\Rightarrow$ (P2 left) is clear. 

(P2 left) $\Rightarrow$ (P1 right): 
The assertion $(N' \subset M^n)  \subset \mathcal{C}$ means that there is a commutative diagram (for $n$ even, say) in which the diagonal sequences are exact:
\[ 
\xymatrix@!@C=0.5em@R=1.0em{   &&&&& 0 \ar[ld] &&&& 0 \ar[ld]  \\
0\ar[rd] &&  && N''''  \ar[ld]  \ar[rd] &&&& N''   \ar[ld]  \ar[rd]  \\
 & M^{\alpha_{n}} \ar@{=}[rd] & &   \cdots  \ar[ll]  \ar[ld] \ar@{-}[r]    \ar[ld]  \ar@{-}[l] \ar[rr] & & M^{\alpha_2} \ar[rd] & & M^{\alpha_1} \ar[rr] \ar[ll]  \ar[ld] & & M^{n} \ar[rd] &&& \\
&&M^{\alpha_n} \ar[rd] &&&&N''' \ar[rd]&&&&N \ar[rd]& \\
&&& 0 &&&&0&&&&0
  }\]
  with $N' = \im(N'' \rightarrow M^n)$. 
Then consider the diagram
\begin{equation*} \xymatrix{ 
\coker \ar@{<<-}[d]  \ar@{<-}[rrd]^{\gamma} & \\ 
(M^n \oplus M^{\alpha_2} \oplus  M^{\alpha_4} \oplus \cdots \oplus M^{\alpha_n} )  \ar@{<-}[d] \ar@{<-}[rr]_-{(\id, 0,  \dots)} && M^n  \\
( M^{\alpha_1} \oplus  M^{\alpha_3} \oplus  \cdots \oplus M^{\alpha_{n-1}}   ) } 
\end{equation*}
One verifies that $N' = \ker(\gamma)$ or, equivalently, that $N = \mathrm{coim}(\gamma)$. 
The other implications are dual. 
\end{proof}

\begin{BEISPIEL}\label{EX1}
\begin{enumerate}
\item Any left or right principal object is principal. 
\item Every semi-simple object is left and right principal. 
\item  If $M$ is a (co)generator in the sense that for {\em any} object there is a surjection $M^k \twoheadrightarrow N$ (resp.\@ injection $N \hookrightarrow M^k$)
then $M$ is  left (resp.\@ right) principal.
\end{enumerate}
\end{BEISPIEL}
In Appendix~\ref{APDX} a machinery for proving principality  for
non semi-simple objects by induction is developed.  In Section~\ref{SECT1MOT} we discuss examples in the category of 1-motives.

\begin{LEMMA}
Let $M$ be an object of an Abelian category.
The class $\mathcal{C}$ of subobjects of  $M^m$ (for varying $m$) as in (P2 left) is closed under intersection and
sum. 
\end{LEMMA}
There is an obvious dual assertion whose formulation we leave to the reader. 
\begin{proof}
For some $n\ge 0$, consider the morphisms
\[ \xymatrix{ M^n \ar[rr]^-{s=(\id, \id)} & & M^n \oplus M^n \ar[rr]^-{p=(\id, \id)} & & M^n. } \]
Given subobjects  $N_1, N_2$ in $M^n$, we have
\[ N_1 \cap N_2 = s^{-1}(N_1 \oplus N_2) \quad  \text{and} \quad N_1 + N_2 = p(N_1 \oplus N_2).   \]
Therefore it suffices to see that $N_1 \oplus N_2 \subset M^{2n}$ is in $\mathcal{C}$. However, $N_1$ and  $N_2$ can be successively build be images and preimages of objects under morphisms $M^n \rightarrow M^m$. One can then take the direct sum of the whole resulting sequences (cf.\@ the proof of (P2 left) $\Rightarrow$ (P1 right) above). 
\end{proof}

\begin{PROP}\label{PROPWP}
There is a surjection
\[  H_B(M) \otimes_{\End(M)^{\op}} H_{dR}(M)^\vee \twoheadrightarrow P(M) \]
which is an isomorphism if $M$ is principal. In this case $P^2(M) = P^3(M) = \cdots = P(M)$. 
\end{PROP}

For an alternative description, observe that any $K$-isomorphism $H_B(M) \otimes_\Q K \cong H_{dR}(M)$ induces an isomorphism
\begin{equation}\label{eqTRANSP}
H_B(M) \otimes_{\End(M)^{\op}} H_{dR}(M)^\vee \cong \End(H_{dR}(M)) / [\End(M), \End(H_{dR}(M))]. 
\end{equation}

\begin{proof}
Let $e \in \End(M)$ and
consider the exact sequence
\[ \xymatrix{ M \ar[r]^-{(e, \id)} & M \oplus M \ar[r]^-{(\id,-e)} & M. }  \]
This gives relations
\[   \sigma \otimes e^\vee(\omega) -  e(\sigma) \otimes \omega = 0  \]
in $P^2(M)$ for all $\sigma \in H_B(M)$ and $\omega \in H_{dR}(M)^\vee$. 

(Seeing $\alpha := \sigma \otimes \omega$ as elements in $ \End(H_{dR}(M))$, and using an isomorphism as in (\ref{eqTRANSP}), we can write this as:
\[ [e, \alpha] = e \alpha - \alpha e \]
Since the relations are obviously closed under sum, we may take $\alpha \in \End(M)$ arbitrary.)

Using Theorem~\ref{SATZQ2} we are left to show that any relation in $P^\infty(M)$ is generated by those if $M$ is principal. 
Consider an exact sequence
\[ \xymatrix{ 0 \ar[r] & N' \ar[r]^{} & M^n \ar[r] & N  \ar[r] & 0, }  \]
$\sigma=(\sigma_i) \in H_B(N')$, and $\omega=(\omega_i) \in H_{dR}(N)^\vee$.
By assumption we find a commutative diagram
\[ \xymatrix{
0 \ar[d] \\
 \ker \ar[d]_{\beta} \ar[r]^{\alpha} & M^n \ar[r] & N \ar[r] & 0 \\
 M^l \ar[ru]_-{(e_{ij})} \ar[d]_-{(f_{ik})}  \\
M^m } \]
such that $N' = \im(\alpha)$ and such that the vertical sequence is also exact. Choose a lift $\widetilde{\sigma} = (\widetilde{\sigma}_1, \dots \widetilde{\sigma}_l)$ in $H_B(\ker)$ of $\sigma$ under $\alpha$.
We have in $H_B(M) \otimes_{\End(M)^{\op}} H_{dR}(M)^\vee$ the relation
\[ \sum_{\substack{i=1..l\\j=1..n}} (\widetilde{\sigma}_i  \otimes e_{ij}^\vee (\omega_j)  -  e_{ij}(\widetilde{\sigma}_i)  \otimes \omega_j ) = 0. \]
However there exist $(\omega'_k) \in H_{dR}(M^m)^\vee$ such that
\[ \sum_{j=1}^n e_{ij}^{\vee} (\omega_j)    = \sum_{k=1}^m  f_{ik}^{\vee} (\omega'_k)  \]
because $\beta^\vee (\sum_j e_{ij}^{\vee} (\omega_j) ) = 0$.
 In $H_B(M) \otimes_{\End(M)^{\op}} H_{dR}(M)$ there is also the relation
\[  \sum_{\substack{i=1..l\\k=1..m}} (\widetilde{\sigma}_i  \otimes f_{ik}^\vee(\omega'_k)  -  f_{ik} (\widetilde{\sigma}_i)  \otimes \omega_k' ) = 0. \]
Now 
\[ \sum_{i=1}^l f_{ik}(\widetilde{\sigma}_i) = 0  \]
because $\widetilde{\sigma} \in H_B(\ker)$.
Putting everything together we get 
\[ \sum_{j=1}^n \left(\sum_{i=1}^l e_{ij}  \widetilde{\sigma}_i \right)  \otimes \omega_j  =  \sum_{j=1}^n \sigma_j \otimes \omega_j  = 0.  \]
Those are precisely the relations that hold in $P^\infty(M)$.
Hence all relations in $P^\infty(M)$ are accounted for already in $H_B(M) \otimes_{\End(M)^{\op}} H_{dR}(M)^\vee$ and  already in $P^2(M)$.
\end{proof}
\begin{FRAGE}
It seems likely that the class of principal objects is the most general for which the formal periods can be described by endomorphisms as in Proposition~\ref{PROPWP}. The author did not investigate this further though. 
\end{FRAGE}

\section{Application to 1-motives}\label{SECT1MOT}

\begin{PAR}
In this section we consider the $\Q$-linear category of Deligne 1-motives \cite{Del74} over $\Qbar$ equipped with their natural fiber functors $(\cat{1-Mot${}_\Qbar$}, H_{B}, H_{dR})$. Recall that a Deligne 1-motive over $\Qbar$ is a morphism
\[ [ L \rightarrow G ] \]
where $G$ is a semi-Abelian variety defined over $\Qbar$ and $L = \Z^n$ for some $n\ge 0$. As semi-Abelian variety $G$ sits in an exact sequence of $\Qbar$-algebraic groups
\[ \xymatrix{ 0 \ar[r] & T \ar[r] & G \ar[r] & A \ar[r] & 0 }  \]
where $A$ is an Abelian variety and $T=(\Gm)^m$ for some $m \ge 0$. 
The $\Q$-linear category of semi-Abelian varieties is a full subcategory of the category of 1-motives interpreting $G$ as the 1-motive $[0 \rightarrow G]$, denoted simply by $G$ again. 
For a finite dimensional $\Q$-vector space $L$ we also denote by $L$ the 1-motive $[L \rightarrow 0] := [\Z \rightarrow 0] \otimes L$. 
\end{PAR}

\begin{PAR}
The morphisms in $\cat{1-Mot${}_\Qbar$}$ are obvious commutative diagrams tensored with $\Q$ to make the category $\Q$-linear. 
The fiber functors are equipped with a weight filtration as in \ref{PARWEIGHT}. Furthermore there is a weight filtration for each 1-motive $M$ in the category $\cat{1-Mot${}_\Qbar$}$ itself, given by
\[ W^{i}(M) = \begin{cases} M & i \ge 0 \\ G & i = -1 \\ T  & i = -2 \\ 0 & i \le -3  \end{cases}\]
We have the Tate motives
\[ \Q = [\Z \rightarrow 0] \qquad \Q(1) =  \Gm  \]
and the graded object of a 1-motive is given by
\[ \mathrm{gr}_W^{i}(M) = \begin{cases} L_\Q \cong \Q^n & i = 0 \\ A & i = -1 \\ T \cong \Q(1)^m  & i = -2  \end{cases}\]
\end{PAR}

We adopt from \cite{HW18} the following notation:
\begin{DEF}A 1-motive is {\bf reduced} if $\Hom(G, \Q(1)) = 0$ and $\Hom(\Q, M/W^{-2}M) = 0$.
A 1-motive is called {\bf saturated} if it is reduced and we have $\End(M) = \End(G) = \End(A)$. 
\end{DEF}

In \cite[Chapter 2]{HW18} it is shown that any 1-motive is a subobject of the direct sum of a saturated 1-motive and a motive of Baker type (i.e.\@ a 1-motive with $A = 0$). 
We will concentrate on the case of saturated 1-motives. 

Consider the subcategories $\mathcal{A}_0 = \cat{1-Mot${}_\Qbar\ \{\le\!-1\}$}$ of 1-motives in weight $-1, -2$ (i.e.\@ semi-Abelian vareities) and $\mathcal{A}_1 = \cat{1-Mot${}_\Qbar\ \{0\}$} \cong \cat{f.d.-$\Q$-Vect}$. 
Then every 1-motive $M = [L \rightarrow G]$ yields an admissible exact sequence in the sense of Definition~\ref{DEFADMISSIBLE}:
\[ \xymatrix{ 0 \ar[r] & G \ar[r] & M  \ar[r] & L_\Q \ar[r] & 0. }  \]
Furthermore the sub-1-motive $M' = [L' \rightarrow G']$ is a universal extension of $G' \subset G$ in the sense of Lemma~\ref{LEMMAUNIVERSAL} if and only if the square
\[ \xymatrix{  L' \ar[r] \ar@{^{(}->}[d] & G' \ar@{^{(}->}[d]  \\
L \ar[r] & G
} \]
is Cartesian. One directly sees that universal extensions exist. Similarly for semi-Abelian schemes
\[ \xymatrix{ 0 \ar[r] & T \ar[r] & G  \ar[r] & A \ar[r] & 0 }  \]
a quotient $G \twoheadrightarrow G'$ is a universal lift of the quotient $A \twoheadrightarrow A'$ if and only if the diagram
\[ \xymatrix{  X^*(T') \ar[r] \ar@{^{(}->}[d] & (A')^\vee \ar@{^{(}->}[d]  \\
X^*(T) \ar[r] & A^\vee
} \]
is Cartesian.

\begin{PROP}\label{PROP1MOT}
If $M$ is a saturated 1-motive  then $M \oplus \Q \oplus \Q(-1)$ is principal.
\end{PROP}

\begin{proof}
The exact sequence
\[ \xymatrix{ 0 \ar[r] & T \ar[r] & G  \ar[r] & A \ar[r] & 0 }  \]
is admissible w.r.t.\@ $\mathcal{A}_0 = \cat{1-Mot${}_\Qbar\ \{-2\}$}$ and $\mathcal{A}_1 = \cat{1-Mot${}_\Qbar\ \{-1\}$} \cong \cat{Ab.var${}_\Qbar$}$.
Note that $\mathcal{A}_0 \cong \cat{f.d.-$\Q$-Vect}$ consists just of powers of $\Q(1)$. 
By Lemma~\ref{LEMMASATSS} the above sequence is right saturated because $\End(G)=\End(A)$ is semi-simple. 
Furthermore $A$ and $T$ are semi-simple and thus principal. 
Since $\Hom(G, \Q(1))=0$, Proposition~\ref{PROPSATURATEDPRINCIPALVAR} implies that $G \oplus \Q(1)$ is principal. 

The exact sequence
\[ \xymatrix{ 0 \ar[r] & G  \ar[r] & M  \ar[r] & L_\Q \ar[r] & 0 }  \]
is admissible w.r.t.\@ $\mathcal{A}_0 = \cat{1-Mot${}_\Qbar\ \{\le\!-2\}$} \cong \cat{Semi.ab.var${}_\Qbar$}$ and $\mathcal{A}_1 = \cat{1-Mot${}_\Qbar\ \{0\}$} \cong \cat{f.d.-$\Q$-Vect}$.
It is left saturated by Lemma~\ref{LEMMASATSS} because $\End(M)=\End(G)$ is semi-simple.
The exact sequence
\[ \xymatrix{ 0 \ar[r] & G \oplus \Q(1) \ar[r] & M \oplus \Q(1) \ar[r] & L_\Q \ar[r] & 0 }  \]
is again left saturated by Lemma~\ref{LEMMASATSUM} because $\Hom(\Q, \coker(\Q(1), M)) = \Hom(\Q, M/T) = 0$ by assumption. 
Furthermore,  since also $\Hom(\Q, M)=0$, the 1-motive
$M \oplus \Q(1) \oplus \Q$ is principal by Proposition~\ref{PROPSATURATEDPRINCIPALVAR}.
\end{proof}

\begin{PAR}
Let $M$ be saturated and assume for simplicity that $\mathrm{gr}^0(M) \not= 0$ and $\mathrm{gr}^{-2}(M) \not= 0$.
The latter implies that $P(M) = P(M')$ for $M' := M \oplus \Q \oplus \Q(1)$.
Propositions \ref{PROP1MOT} and \ref{PROPWP} imply therefore (using any identification $H(M'):=H_B(M')\otimes \Qbar \cong H_{dR}(M')$):
\[ P(M) = P(M') \cong \End(H(M')) / [\End(M'),\End(H(M'))] \]

Denote $B := \End(M) = \End(G) = \End(A)$ and observe that $B$ acts on $\mathrm{gr}^{0}M=L_\Q$ and $\mathrm{gr}^{-2}M = T$.

The weight filtration on $H_{B}(M)$, and $H_{dR}(M)$, respectively, induces a filtration
\[ P(M) = P^{0}(M) \supseteq P^{-1}(M) \supseteq P^{-2}(M) \supseteq  P^{-3}(M) = 0. \]
We have furthermore using saturatedness
\begin{eqnarray*}
 \End(M') &=& \underbrace{\End(A)}_{B} \oplus \underbrace{\End(\Q)}_{\Q} \oplus \underbrace{\End(\Q(1))}_{\Q}  \\
    && \oplus\underbrace{\Hom(\Q(1), T)}_{X_*(T)_\Q} \oplus \underbrace{\Hom(L, \Q)}_{L^\vee_\Q}
\end{eqnarray*}
and after a short calculation of the commutators $[\End(M'),\End(H(M'))]$ we get
\end{PAR}

\begin{KOR}\label{CORPERMOTIVES}
If $M$ is a saturated 1-motive with $\mathrm{gr}^0(M) \not= 0$ and $\mathrm{gr}^{-2}(M) \not= 0$, then  denoting $B:=\End(M)=\End(G)=\End(A)$,
the non-zero pieces of the graded object $\mathrm{gr} P(M)$ are given by:
\begin{eqnarray*}
 \mathrm{gr}^i P(M) &\cong&  \begin{cases} \End(H(\Q)) \oplus \End(H(\Q(1))) \oplus \End_{B}(H(A)) & i=0  \\
   \Hom_{B}(H(T), H(A)) \oplus \Hom_{B}(H(A), H(L_\Q)) & i=-1   \\
 \Hom_{B}(H(T), H(L_\Q))   & i=-2 \end{cases}
\end{eqnarray*}
\end{KOR}
Counting dimensions one gets back the formulas obtained in \cite[Sections 9.2, 9.10, 9.14]{HW18}.

\begin{PAR}
As another illustration of the method consider
 a 1-motive of Baker type  $M=[L \rightarrow \Gm\otimes X]$ (with $L$ and $X$ non-zero $\Q$-vector spaces) although in this simple case it is easier to examine the space of formal periods directly.
By Proposition~\ref{PROPBASICPROP}, 4., we have
\[ P(M) = P(\widetilde{M}) \]
with $\widetilde{M}$ of the form:
\[ 0 \rightarrow \Q(1) \rightarrow \widetilde{M} \rightarrow \Hom(X, L) \rightarrow 0.  \]
$\widetilde{M}$ has a maximal submotive $N \subset \widetilde{M}$ of weight 0
and we get an exact sequence
\[ 0 \rightarrow \Q(-1) \rightarrow \widetilde{M}/N \rightarrow \Hom(X, L)/N \rightarrow 0  \]
which is left saturated by Lemma~\ref{LEMMASATSS} and thus by Proposition~\ref{PROPSATURATEDPRINCIPALVAR}  the 1-motive 
 $\widetilde{M}/N \oplus \Q$ is principal. Therefore 
 \[ P(\widetilde{M}/N \oplus \Q) \cong \End(H(\Q(1))) \oplus \End(H(\Q)) \oplus \Hom(X, L)/N.   \]
\end{PAR}

\appendix

\section{The Yoga of principality and saturatedness}\label{APDX}

\begin{PAR}
Let the setting be as in \ref{PARTUPLE}.
In this section let $\mathcal{A}_0, \mathcal{A}_1$ be Abelian subcategories of $\mathcal{A}$ such that 
\begin{equation} \label{eqorth}
\Hom(M_0, M_1) = \Hom(M_1, M_0) = 0 \quad  \text{ for all pairs } (M_0, M_1) \text{ with } M_i \in \mathcal{A}_i.
\end{equation}
We write this more compactly as 
$\Hom(\mathcal{A}_0, \mathcal{A}_1) = \Hom(\mathcal{A}_1, \mathcal{A}_0) = 0$
and also use the notation $\Hom(M, \mathcal{A}_0)=0$, say,  to express $\Hom(M, M_0)=0$ for all $M_0 \in \mathcal{A}_0$.
\end{PAR}

\begin{DEF}\label{DEFADMISSIBLE}
An exact sequence
\[ \xymatrix{  0 \ar[r] & M_0 \ar[r] & M \ar[r] & M_1  \ar[r] & 0    } \]
with $M_i \in \mathcal{A}_i$ is called an {\bf admissible} exact sequence.
\end{DEF}

\begin{PAR}
 In this section we investigate criteria for the principality of $M \oplus M_0$ (resp.\@ $M \oplus M_1$) supposing
that $M_0$ and $M_1$ are principal. Note that it is almost never the case that $M$ itself is principal unless the sequence splits.
Note that we have
\[ P(M) = P(M \oplus M_0) = P(M \oplus M_1) \]
by Proposition~\ref{PROPBASICPROP}
hence knowing that $M \oplus M_0$ (resp.\@ $M \oplus M_1$) is principal allows to compute $P(M)$ using endomorphisms of $M \oplus M_0$ (resp.\@ of $M \oplus M_1$). 
Induction may be applied. 
\end{PAR}

\begin{PAR}
Every subobject $M' \subset M$ induces an injection of admissible exact sequences: 
\[ \xymatrix{  
0 \ar[r] & M'_0  \ar@{^{(}->}[d] \ar[r] & M'  \ar@{^{(}->}[d] \ar[r] & M_1'  \ar@{^{(}->}[d] \ar[r] & 0  \\
0 \ar[r] & M_0  \ar[r] & M  \ar[r] & M_1  \ar[r] & 0 }  \]
The same is true for quotients. 
Furthermore for two admissible sequences and a given morphism $M \rightarrow N$ 
\[ \xymatrix{  0 \ar[r] & M_0 \ar[r] \ar@{.>}[d] & M \ar[r] \ar[d] & M_1 \ar@{.>}[d]  \ar[r] & 0    \\
  0 \ar[r] & N_0 \ar[r] & N \ar[r] & N_1  \ar[r] & 0    } \]
there is always an induced morphism of exact sequences (because $\Hom(M_0, N_1)=0$ by (\ref{eqorth})) and by the Snake Lemma,
it induces an admissible subsequence of kernels and an admissible quotient sequence of cokernels. 
\end{PAR}

\begin{LEMMA}\label{LEMMAUNIVERSAL}
Let 
\[ \xymatrix{  0 \ar[r] & M_0 \ar[r]^i & M \ar[r]^p & M_1  \ar[r] & 0    } \]
be an admissible exact sequence.
\begin{enumerate}
\item Assume that $\mathcal{A}$ is Artinian. For each subobject $N'_1 \subset M_1$ there exists a universal lift
\[ \xymatrix{  0 \ar[r] & N'_0 \ar[r] & N' \ar[r] & N_1'  \ar[r] & 0    }  \]
with $N' \subset M$ such that every other lift of the subobject contains $N'$. Equivalently, for every quotient $M_1 \twoheadrightarrow N_1$
there is a universal lift
\[ \xymatrix{  0 \ar[r] & N_0 \ar[r] & N \ar[r] & N_1  \ar[r] & 0    }  \]
with $M \twoheadrightarrow N$ such that every other lift of the quotient factors through $N$. 

\item Assume that $\mathcal{A}$ is Noetherian. For each subobject $N'_0 \subset M_0$ there exists a universal extension
\[ \xymatrix{  0 \ar[r] & N'_0 \ar[r] & N' \ar[r] & N_1'  \ar[r] & 0    }  \]
with $N' \subset M$ such that every other extension of the subobject is contained in $N'$. Equivalently, for every quotient $M_0 \twoheadrightarrow N_0$
there is a universal extension
\[ \xymatrix{  0 \ar[r] & N_0 \ar[r] & N \ar[r] & N_1  \ar[r] & 0    }  \]
with $M \twoheadrightarrow N$ such that every other lift of the quotient projects onto $N$. 
\end{enumerate}
The universal lifts are closed under sums (of the subobjects). The universal extensions are closed under intersections (of the subobjects). 
\end{LEMMA}
\begin{proof}
Let 
\[ \xymatrix{  0 \ar[r] & M'_0 \ar[r] & M' \ar[r] & M_1'  \ar[r] & 0    }  \]
\[ \xymatrix{  0 \ar[r] & M''_0 \ar[r] & M'' \ar[r] & M_1''  \ar[r] & 0    }  \]
be two exact subsequences of the given admissible exact sequence. Consider the diagram
\[ \xymatrix{  
0 \ar[r] & M'_0 \cap M''_0  \ar@{^{(}->}[d] \ar[r] & M' \cap M''  \ar@{^{(}->}[d] \ar[r] & M_1' \cap M_1''  \ar@{^{(}->}[d] \ar[r] & \cdots  \\
0 \ar[r] & M'_0 \oplus M''_0  \ar[d] \ar[r] & M' \oplus M''  \ar[d] \ar[r] & M_1' \oplus M_1'' \ar[d] \ar[r] & 0   \\
0 \ar[r] & (M' + M'') \cap M_0 \ar@{->>}[d] \ar[r] & M' + M'' \ar@{->>}[d] \ar[r] & M_1' + M_1'' \ar@{->>}[d]\ar[r] & 0   \\
\cdots \ar[r] & \coker \ar[r] & 0 \ar[r]  & 0
   }  \]
 The snake lemma and the fact that there are no morphisms $M_1' \cap M_1'' \rightarrow \coker$ (because $(M' + M'') \cap M_0 \in \mathcal{A}_0$) imply
 \begin{equation} \label{eqsumintersect}
  p(M' \cap M'') = p(M') \cap p(M'')  \qquad i^{-1}(M' + M'') = i^{-1}(M') + i^{-1}(M'').
  \end{equation}

This implies that for two lifts $N'$, and $N''$, of $N_1'$ the intersection $N' \cap N''$ is also a lift of $N_1'$. 
 Since $\mathcal{A}$ is Artinian the intersection of all lifts exists and clearly satisfies the universal property. 
 The other statement is dual (swapping the two equivalent formulations). 

 If $N'$ and $N''$ are universal lifts of $N_1'$ and $N_1''$, respectively, and if $N'''$ is a subobject with $p(N''')=N_1'+N_1''$ then $(N''') \cap p^{-1}N_1'$ is
 a lift of $N_1'$ (because of (\ref{eqsumintersect})) and thus contains $N'$ by universality. Similarly
 $(N''') \cap p^{-1}N_2'$ contains $N''$. Therefore $N'''$ contains $N' + N''$. Again the other statement is dual.  
\end{proof}

\begin{DEF}
An admissible exact sequence 
\[ \xymatrix{  0 \ar[r] & M_0 \ar[r] & M \ar[r] & M_1  \ar[r] & 0    } \]
is called
\begin{enumerate}
\item {\bf right saturated} if every morphism
$\varphi_1 : M_1^n \rightarrow M_1^{m}$
has a lift $\varphi: M^n \rightarrow M^m$ such that 
\[ \xymatrix{  0 \ar[r] & \ker(\varphi_0)  \ar[r] & \ker(\varphi)  \ar[r] & \ker(\varphi_1)  \ar[r] & 0    }  \]
is the universal lift of $\ker(\varphi_1)$. 

\item{\bf left saturated}  if every morphism
$\varphi_0 : M_0^n \rightarrow M_0^{m}$
has an extension $\varphi: M^n \rightarrow M^m$ such that 
\[ \xymatrix{  0 \ar[r] & \im(\varphi_0)  \ar[r] & \im(\varphi)  \ar[r] & \im(\varphi_1)  \ar[r] & 0    }  \]
is the universal extension of $\im(\varphi_0)$. 
\end{enumerate}
\end{DEF}

\begin{LEMMA}\label{LEMMASATSS}
Let 
\[ \xymatrix{  0 \ar[r] & M_0 \ar[r]^i & M \ar[r]^p & M_1  \ar[r] & 0    } \]
be an admissible exact sequence.
\begin{enumerate}
\item
If $\End(M) \rightarrow \End(M_1)$ is split surjective and $\End(M_1)$ is a semi-simple $\Q$-algebra and $\Hom(M, \mathcal{A}_0)=0$ 
then the sequence is right saturated.

\item
If $\End(M) \rightarrow \End(M_0)$ is split surjective and $\End(M_0)$ is a semi-simple $\Q$-algebra and $\Hom(\mathcal{A}_1, M)=0$
then it the sequence is left saturated.
\end{enumerate}
\end{LEMMA}
\begin{proof}We prove assertion 1. The other is dual. 
Consider a morphism $\varphi_1: M_1^n \rightarrow M_1^m$ and the lift given by the splitting of $\End(M) \rightarrow \End(M_1)$:
\[ \xymatrix{  
0 \ar[r] & M_0^n \ar[d]^{\varphi_0} \ar[r] & M^n\ar[d]^{\varphi} \ar[r] & M_1^n  \ar[r] \ar[d]^{\varphi_1} & 0    \\
0 \ar[r] & M_0^m \ar[r] & M^m \ar[r] & M_1^m  \ar[r] & 0    
} \]
Because of the semi-simplicity of $\End(M_1)$ we get $\varphi'_1 : M_1^m \rightarrow M_1^n$ such that $\varphi_1 \varphi'_1$ is idempotent and thus 
also a lift $\varphi'$ such that $\varphi \varphi'$ is idempotent. This yields a decomposition of $M^m$ of the following form:
\[ \xymatrix{  
0 \ar[r] & \im(\varphi_0) \oplus M_0' \ar[r] & \im(\varphi) \oplus M' \ar[r] & \im(\varphi_1) \oplus M_1'  \ar[r] & 0   . 
} \]
Given any submodule $N' \subset M^m$ such that $p(N')=\im(\varphi_1)$, intersecting it with $\im(\varphi)$,
we get the following diagram with exact rows and columns
\[ \xymatrix{  
0 \ar[r] & N_0' \cap \im(\varphi) \ar@{^{(}->}[d]\ar[r] & N' \cap \im(\varphi) \ar@{^{(}->}[d]\ar[r] & \im(\varphi_1)  \ar[r] \ar@{=}[d] & 0    \\
0 \ar[r] & \im(\varphi_0) \ar[r] \ar@{->>}[d]  & \im(\varphi) \ar[r]  \ar@{->>}[d] & \im(\varphi_1)  \ar[r]  & 0    \\
 & A_0 \ar@{=}[r] & A_0
} \]
where $A_0$ is defined as the cokernel. 
Note that  $N' \cap \im(\varphi) \to \im(\varphi_1)$ is surjective because of (\ref{eqsumintersect}). This yields a morphism
\[ M^m \twoheadrightarrow \im(\varphi) \twoheadrightarrow A_0 \]
which is zero by assumption. Therefore $\im(\varphi) \subset N'$. 
\end{proof}

\begin{PAR}
Assuming that $\mathcal{A}$ is Noetherian (resp.\@ Artinian), 
for any two objects $M, N \in \mathcal{A}$, we denote 
\begin{eqnarray*}
 \coker(N_0,M) &:=& M/\sum_{\alpha \in \Hom(N_0, M)} \im(\alpha) \\
 \ker(M,N_1) &:=& \bigcap_{\alpha \in \Hom(M, N_1)} \ker(\alpha).
\end{eqnarray*}
\end{PAR}

\begin{LEMMA}\label{LEMMASATSUM}Consider two admissible and left (resp.\@ right) saturated exact sequences:
\[ \xymatrix@R=1em{  0 \ar[r] & M_0 \ar[r] & M \ar[r] & M_1  \ar[r] & 0    \\
  0 \ar[r] & N_0 \ar[r] & N \ar[r] & N_1  \ar[r] & 0    } \]
Their sum 
\[ \xymatrix{  0 \ar[r] & M_0 \oplus N_0 \ar[r] & M \oplus N \ar[r] & M_1 \oplus N_1  \ar[r] & 0    } \]
is left (resp.\@ right) saturated provided that the properties (L1--L3) (resp.\@ (R1--R3)) hold:

\begin{multicols}{2}
\begin{enumerate}
\item[(L1)] $\Hom(M_0, N_0) = 0$.
\item[(L2)]  $\Hom(N, M_0) \rightarrow \Hom(N_0, M_0)$ surjective\footnote{hence bijective, because $\Hom(N_1, M_0)=0$}.
\item[(L3)] $\Hom(\mathcal{A}_1, \coker(N_0,M)) = 0$. 
\item[(R1)] $\Hom(N_1, M_1) = 0$.
\item[(R2)]  $\Hom(M_1, N) \rightarrow \Hom(M_1, N_1)$ surjective\footnote{hence bijective, because $\Hom(M_1, N_0)=0$}.
\item[(R3)] $\Hom(\ker(M, N_1), \mathcal{A}_0) = 0$. 
\end{enumerate}
\end{multicols}
\end{LEMMA}
\begin{proof}
We concentrate on the left case. The other case is dual.
By (L1) we have 
\[ \End(M_0 \oplus N_0) = \End(M_0) \oplus \Hom(N_0, M_0) \oplus \End(N_0). \]
Given a morphism $\varphi_0: M_0^n \oplus N_0^n \rightarrow M_0^m \oplus N_0^m$ for some $n$ and $m$ we may thus choose a lift $\varphi$ in
\[  \Hom(M^n, M^m) \oplus \Hom(N^n, M^m) \oplus \Hom(N^n, N^m) \subset \Hom(M^n \oplus N^n, M^m \oplus N^m) \]
in such a way that the first and last component are the distinguished lifts given by the definition of left saturated and the middle component is the one determined by
the bijection of (L2).
This yields a morphism of exact sequences
\[ \xymatrix{  
0 \ar[r] & M_0^n \oplus N_0^n \ar[d]^{\varphi_0} \ar[r] & M^n \oplus N^n \ar[d]^{\varphi} \ar[r] & M_1^n \oplus N_1^n  \ar[r] \ar[d]^{\varphi_1} & 0    \\
0 \ar[r] & M_0^m \oplus N_0^m \ar[r] & M^m \oplus N^m \ar[r] & M_1^m \oplus N_1^m  \ar[r] & 0    
} \]
and the image is the sum of the images of the morphisms restricted to  $M^n$ and $N^n$. It suffices hence to see that the images of these restrictions are universal (by Lemma~\ref{LEMMAUNIVERSAL}).
Since the $\Hom(M^n, N^m)$ component of $\varphi$ is zero by choice, the morphism from $M^n$ goes to $M^m$ and has thus universal image by assumption.
It suffices to see that the image of the morphism 
\[ \xymatrix{  
0 \ar[r] & N_0^n \ar[d]^{(\varphi_{M,0},\, \varphi_{N,0})} \ar[r] & N^n \ar[d]^{(\varphi_{M},\, \varphi_{N})} \ar[r] &  N_1^n  \ar[r] \ar[d]^{(0,\, \varphi_{N,1})} & 0    \\
0 \ar[r] & M_0^m \oplus N_0^m \ar[r] & M^m \oplus N^m \ar[r] & M_1^m \oplus N_1^m  \ar[r] & 0    
} \]
is the universal extension.
Let $K'$ be any subobject of $M^m \oplus N^m$ extending $\im(\varphi)$:
\[ \xymatrix{  
0 \ar[r] & \im(\varphi)  \ar[r] & K' \ar[r] &  K'_1  \ar[r]& 0.
} \]
 Projecting to $M^m$ gives an exact sequence
\[ \xymatrix{  
0 \ar[r] & \im(\varphi_{M,0})  \ar[r] & p_M(K') \ar[r] &  p_M(K'_1)  \ar[r]& 0.
} \]
This gives a non-trivial morphism $p_M(K'_1) \rightarrow p_M(K')/\im(\varphi_{M,0}) \rightarrow M/\im(\varphi_{M,0})$
which cannot become zero in $M/\im(N_0,M)$ because of $\Hom(\mathcal{A}_1, \mathcal{A}_0) = 0$. Contradiction to (L3). 
Hence we have $p_M(K')=\im(\varphi_{M,0})=\im(\varphi_{M})$.  
This shows that $K'$ sits in an exact sequence of the form
 \[ \xymatrix{  
0 \ar[r] & \im(\varphi)  \ar[r] & K' \ar[r] &  (0 \oplus N_1')  \ar[r]& 0  .  
} \]
Projecting to $N^m$ and using that $\varphi_N$ has universal image we get $N'_1 \subset \im(\varphi_{N,1})$ and hence $K' \subset \im(\varphi)$.
\end{proof}

\begin{PROP}\label{PROPSATURATEDPRINCIPAL}
Let 
\[ \xymatrix{  0 \ar[r] & M_0 \ar[r] & M \ar[r] & M_1  \ar[r] & 0    } \]
be an admissible exact sequence with $M_0 \in \mathcal{A}_0$ and  $M_1 \in \mathcal{A}_1$  both principal. Then
\begin{enumerate}
\item if the sequence is right saturated and $\Hom(M, \mathcal{A}_0)=0$ then $M \oplus M_0$ is principal.
\item if the sequence is left saturated and $\Hom(\mathcal{A}_1, M)=0$ then $M \oplus M_1$ is principal. 
\end{enumerate}
\end{PROP}

\begin{proof}We prove assertion 1. Assertion 2.\@ is dual. 
We apply Lemma~\ref{LEMMASATSUM} to the given exact sequence and the exact sequence
\[ \xymatrix{  0 \ar[r] & N_0:=M_0 \ar[r] & N:=M_0 \ar[r] & N_1:=0  \ar[r] & 0.    }  \]
The assertions (R1--R2) are trivial and (R3) boils down to the assumption $\Hom(M, \mathcal{A}_0)=0$. 
This shows that
\[ \xymatrix{  0 \ar[r] & M_0^2 \ar[r] & M \oplus M_0 \ar[r] & M_1  \ar[r] & 0    } \]
is again right saturated.
Consider a subobject for some $m\ge 0$:
\[ \xymatrix{  
0 \ar[r] & N_0' \ar[r] \ar@{^{(}->}[d] & N' \ar[r] \ar@{^{(}->}[d] & N_1'  \ar[r] \ar@{^{(}->}[d] & 0  \\
0 \ar[r] & (M_0^2)^m \ar[r] & (M \oplus M_0)^m \ar[r] & M_1^m  \ar[r] & 0  
  } \]
Since $M_1$ is principal there is a diagram
\begin{equation*} \xymatrix{ 
\ker(\beta_1) \ar@{^{(}->}[d]& \\ 
M^k_1 \ar@{->}[d]^{\beta_1} \ar[r]^{\alpha_1} & M_1^m  \\
M^l_1 } \end{equation*}
such that $N_1' = \alpha_1(\ker(\beta_1))$. Since $\End(M \oplus M_0) \rightarrow \End(M_1)$ is surjective, we can arrange a lift $\alpha$ 
(possibly by increasing $k$) such that
\[ \alpha: (M \oplus M_0)^k \rightarrow (M \oplus M_0)^m  \]
satisfies $(M_0^2)^m \subset \im(\alpha)$ and hence $N' \subset \im(\alpha)$. (Observe that we are free to choose $\alpha$
on $M_0^k$.)
Furthermore we may choose a lift $\beta$ of $\beta_0$ in such a way that $\ker(\beta)$ is the universal lift of $\ker(\beta_1)$. We get a diagram with exact rows and columns: 
\[ \xymatrix{  
0 \ar[r] & \ker(\beta_0) \ar[r] \ar@{^{(}->}[d] & \ker(\beta) \ar[r] \ar@{^{(}->}[d] & \ker(\beta_1)  \ar[r] \ar@{^{(}->}[d] & 0  \\
0 \ar[r] & (M_0^2)^k \ar[d]^{\beta_0} \ar[r] & (M \oplus M_0)^k \ar[r]^-p \ar[d]^{\beta}  & M_1^k  \ar[r] \ar[d]^{\beta_1} & 0   \\
0 \ar[r] & (M_0^2)^l \ar[r]^-i & (M \oplus M_0)^l \ar[r] & M_1^l  \ar[r] & 0   \\
  } \]
 
Consider $N'' :=  \alpha^{-1}(N') \cap p^{-1}(\ker_1)$. The surjectivity of $\alpha_0$ implies that we have epimorphisms
\[   \alpha^{-1}(N') \cap p^{-1}(\ker(\beta_1)) \twoheadrightarrow  N' \times_{N_1'} \ker(\beta_1) \twoheadrightarrow N'. \]
Therefore $\alpha(N'') = N'$. Since  $\ker(\beta) \subset N''$ because of universality, there is a subdiagram with exact rows and columns:
\[ \xymatrix{  
0 \ar[r] & \ker(\beta_0)  \ar[r] \ar@{^{(}->}[d] &  \ker(\beta)  \ar[r] \ar@{^{(}->}[d] & \ker(\beta_1)  \ar[r] \ar@{=}[d] & 0  \\
0 \ar[r] & N'' \cap (M_0^2)^k  \ar@{->>}[d] \ar[r] & N'' \ar[r] \ar@{->>}[d] & \ker(\beta_1) \ar[r] & 0   \\
& N_0'' \ar[r]^{\sim} & i(N_0'')  \\
  } \]
Now $N_0'' \subset (M_0^{2})^l$ lies in the class $\mathcal{C}$ (cf.\@ (P2 left) of Definition~\ref{DEFWP}) for the object $M_0$ because $M_0$ is principal. Hence also $0 \oplus N_0''  \subset (M \oplus M_0)^{2l}$ lies in the class $\mathcal{C}$ for the object $M_0 \oplus M$. The morphism 
\[ (0, i): (M \oplus M_0)^{2l} \rightarrow (M \oplus M_0)^{l} \]
 maps $0 \oplus N_0''$ to $i(N_0'')$. 
Therefore also $N'' = \beta^{-1}(i(N_0'')) \subset (M \oplus M_0)^k$ lies in $\mathcal{C}$ and hence
also $N' = \alpha(N'') \subset (M \oplus M_0)^m$ lies in $\mathcal{C}$.
\end{proof}

With slight modifications of the proof one has also the following variant:

\begin{PROP}\label{PROPSATURATEDPRINCIPALVAR}
Let 
\[ \xymatrix{  0 \ar[r] & M_0 \ar[r] & M \ar[r] & M_1  \ar[r] & 0    } \]
be an admissible exact sequence with $M_0 \in \mathcal{A}_0$ and  $M_1 \in \mathcal{A}_1$  both principal. Then
\begin{enumerate}
\item if the sequence is right saturated and $\Hom(M, \mathcal{A}_0)=0$ and $A_0$ is a generator of $\mathcal{A}_0$ then $M \oplus A_0$ is principal.
\item if the sequence is left saturated and $\Hom(\mathcal{A}_1, M)=0$ and $A_1$ is a cogenerator of $\mathcal{A}_1$ then, $M \oplus A_1$ is principal. 
\end{enumerate}
\end{PROP}

\begin{PAR}\label{PARWEIGHT}
Proposition~\ref{PROPSATURATEDPRINCIPAL} can be applied in the following situation. An Abelian category $(\mathcal{A}, H_{B}, H_{dR})$ with fiber functors as above  is said to be {\bf equipped with a weight filtration} if  $H_{B}$ factors
\[ \mathcal{A} \rightarrow \cat{f.d.-filt-$\Q$-Vect } \rightarrow \cat{f.d.-$\Q$-Vect } \] 
through $\cat{f.d.-filt-$\Q$-Vect } $, the {\em exact} category of filtered vector spaces and strict morphisms,
and similarly for $H_{dR}$. 

An object $\mathcal{M}$ is said to be pure of weight $n$ if $H_B(M)$ is pure of weight $n$. 
An exact sequence
\[ \xymatrix{  0 \ar[r] & M_0 \ar[r] & M \ar[r] & M_1  \ar[r] & 0    } \]
in which $M_0$ and $M_1$ are pure of different weights (or more generally such that different weights occur in $\mathrm{gr} H_B$) is automatically admissible for the appropriate subcategories.
This allows for instance to prove by induction that $M \oplus W^{-1}M \oplus W^{-2}M \oplus \cdots$ is principal, starting from pure objects --- which are (conjecturally) semi-simple in concrete categories of mixed motives --- if all assumptions are satisfied. In Section~\ref{SECT1MOT} this is exploited for 1-motives recovering results of \cite{HW18}.
\end{PAR}

\newpage

\bibliographystyle{abbrvnat}
\bibliography{paper6}

\end{document}